\documentclass[12pt]{amsart}

\usepackage{amssymb,amsmath}
\usepackage{hyperref}
\usepackage[all]{xy}
\usepackage{empheq,fancybox}

\theoremstyle{definition}
\newtheorem{theorem}{Theorem}

\newtheorem{definition}{Definition}
\theoremstyle{remark}

\begin{document}

\title{On categories of o-minimal structures}

\author[R. Figueiredo]{Rodrigo Figueiredo}
\thanks{}
\address{Departamento de Matem\'atica, Instituto de Matem\'atica e Estat\'is\-tica, Universidade de S\~ao Paulo, Rua do Mat\~ao 1010, CEP 05508-090, S\~ao Paulo, SP, BRAZIL.}
\email{rodrigof@ime.usp.br}

\author[H. L. Mariano]{Hugo Luiz Mariano}
\address{Departamento de Matem\'atica, Instituto de Matem\'atica e Estat\'is\-tica, Universidade de S\~ao Paulo, Rua do Mat\~ao 1010, CEP 05508-090, S\~ao Paulo, SP, BRAZIL.}
\email{hugomar@ime.usp.br}

%%\author[P. Others]{Possibly Others}

\begin{abstract}
Our aim in this paper is to look at some transfer results in model theory (mainly in the context of o-minimal structures) from the category theory viewpoint. 
\end{abstract}


\maketitle

\tableofcontents

\section{Introduction}

Our aim in this paper is to look at some transfer results in the context of o-minimal structures from the category theory viewpoint. Recall that an \textit{o-minimal structure} $\mathcal{M}$ is an expansion of an ordered set $(|\mathcal{M}|,\leq)$ such that every unary set definable in $\mathcal{M}$ (with parameters in $|\mathcal{M}|$) is a finite union of open intervals and points. For a detailed exposition of this topic, see \cite{vandendries-tame}.

In \cite{berarducci-otero} A. Berarducci and M. Otero point out some transfer results with respect to topological properties from one o-minimal structure to another. Specifically, if $\mathcal{M}$ is an o-minimal expansion of an ordered field and $\varphi$ is a first order formula in the language of the ordered rings, then the following statements concerning the definable subsets $\varphi^{\mathcal{M}}$ and $\varphi^{\mathbb{R}}$ hold: (1) $\varphi^{\mathcal{M}}$ is definably connected if and only if $\varphi^{\mathbb{R}}$ is connected; (2) $\varphi^{\mathcal{M}}$ is definably compact if and only if $\varphi^{\mathbb{R}}$ is compact; (3) there is a natural isomorphism between the homology groups $H_*^{\text{def}}(\varphi^{\mathcal{M}})\cong H_*(\varphi^{\mathbb{R}})$; (4) there is a natural isomorphism between the fundamental groups $\pi^{\text{def}}(\varphi^{\mathcal{M}},x_0)\cong \pi(\varphi^{\mathbb{R}},x_0)$; and assuming that $\varphi^{\mathbb{R}}$ is compact it follows that (5) if $\varphi^{\mathcal{M}}$ is a definable manifold, then $\varphi^{\mathbb{R}}$ is a (topological) manifold; and (6) if moreover $\varphi^{\mathcal{M}}$ is definably orientable, then $\varphi^{\mathbb{R}}$ is an orientable manifold.

In \cite{miller-starchenko}, C. Miller and S. Starchenko prove a dichotomy theorem on o-minimal expansions of ordered groups: 
\\

\noindent\textbf{Fact 1} (Theorem A, \cite{miller-starchenko}). Suppose that $\mathcal{R}$ is an o-minimal expansion of an ordered group $(R,<,+)$. Then exactly one of the the following holds: (a) $\mathcal{R}$ is linearly bounded (that is, for each definable function $f\colon R\to R$ there exists a definable endomorphism $\lambda\colon R\to R$ such that $|f(x)|\leq \lambda(x)$ for all sufficiently large positive arguments $x$); (b) $\mathcal{R}$ defines a binary operation $\cdot$ such that $(R,<,+,\cdot)$ is an ordered real closed field. If $\mathcal{R}$ is linearly bounded, then for every definable $f\colon R\to R$ there exists $c\in R$ and a definable $\lambda\in \{0\}\cup \text{Aut}(R,+)$ with $\lim_{x\to +\infty}[f(x)-\lambda(x)]=c$.
\\

Such a dichotomy on o-minimal expansions of ordered groups is the analogue of the subsequent dichotomy for o-minimal expansions of the real field $\mathbb{R}$, due to C. Miller:
\\

\noindent\textbf{Fact 2} (Theorem and Proposition, \cite{miller}). Let $\mathcal{R}$ be an o-minimal expansion of the ordered field of real numbers $(\mathbb{R}, <, +,\cdot, 0,1)$. If $\mathcal{R}$ is not polynomially bounded (that is, for every definable function $f\colon \mathbb{R}\to \mathbb{R}$ there exists $N\in \mathbb{N}$ such that $|f(x)|\leq x^N$ for all sufficiently large positive $x$), then the exponential function is definable (without parameters) in $\mathcal{R}$. If $\mathcal{R}$ is polynomially bounded, then for every definable function $f\colon \mathbb{R}\to \mathbb{R}$, with $f$ not identically zero for all sufficiently large positive arguments, there exist $c,r\in \mathbb{R}$ with $c\neq 0$ such that $x\mapsto x^r\colon (0,+\infty)\to \mathbb{R}$ is definable in $\mathcal{R}$ and $\lim_{x\to +\infty}f(x)/x^r=c$. 
\\

Both Facts 1 and 2 can be viewed as implied transfer results of o-minimality property from one structure to another (see Section 4) and served as our main motivation for this work. 

%%Chamar atencao para resiltados de trnsferencia em teoria o -minimais

%%caso explicito: topologico, berrducci

%%2 cass implicitos: dicotomnsa

%%Enunciar os 2 teoremas Grupos e coros (pois vaos citar e usar)

\section{Preliminaries}

Recall that a \textit{signature} is a triple $L\mathrel{\mathop:}=(\mathcal{F},\mathcal{R},\text{ar})$, where $\mathcal{F}$ and $\mathcal{R}$ are disjoint sets whose members are called respectively \textit{function symbols} and \textit{predicative symbols} and $\text{ar}\colon \mathcal{F}\cup \mathcal{R}\to \mathbb{N}$ is a function which assigns a nonnegative integer, called \textit{arity}, to every function or predicative symbol. A function or a predicative symbol is said to be \textit{$n$-ary} if its arity is $n$. A $0$-ary function symbol is called a \textit{constant symbol}. The cardinality $\text{card}(L)$ of a signature $L=(\mathcal{F},\mathcal{R},\text{ar})$ is defined to be $\text{card}(\mathcal{F})+\text{card}(\mathcal{R})$.

The first-order \textit{language} of a signature $L$ is the set of all (well formed) terms and formulas arising from $L$, and is denoted by $\mathcal{L}$. If we denote by $\text{Term}(L)$ the set of all $L$-terms, and by $\text{Form}(L)$ the set of all $L$-formulas then $\mathcal{L}=\text{Term}(L)\sqcup \text{Form}(L)$.

Let $\mathcal{L}$ and $\mathcal{L}'$ be two first-order languages. A \textit{language morphism} from $\mathcal{L}$ to $\mathcal{L}'$ is a (set-theoretic) map $H\colon \mathcal{L}\to \mathcal{L}'$ such that $h$ maps terms from $\mathcal{L}$ to terms from $\mathcal{L}'$, and formulas from $\mathcal{L}$ to formulas from $\mathcal{L}'$.

\section{A category of the first-order languages}

Fix a countable set of variable symbols $\text{Var}=\{x_i\,:\,i\in \mathbb{N}\}$.

In what follows we make a brief description of the category $\textbf{FOL}$ of the first-order languages. 

Let $\text{Ob}(\textbf{FOL})$ denote the set of all first-order languages. 

Given two languages $\mathcal{L},\mathcal{L}'\in \text{Ob}(\textbf{FOL})$, with underlying signatures $L=(\cup_{n\geq 0} F_n,\cup_{n\geq 0} R_n)$ and $L'=(\cup_{n\geq 0} F_n',\cup_{n\geq 0} R'_n)$ respectively, the correspondence for each $n\geq 0$
\begin{enumerate}
    \item[(i)] $f\mapsto h(f)$, an $\mathcal{L}'$-term whose variable symbols occurring in it are precisely $x_0,\ldots,x_{n-1}$, $f\in F_n$;
    \item[(ii)] $R\mapsto h(R)$, an $\mathcal{L}'$-atomic formula whose variable symbols occurring in it are precisely $x_0,\ldots,x_{n-1}$, $R\in R_n$.
\end{enumerate}
gives rise to a language morphism $H\colon \mathcal{L}\to \mathcal{L}'$, where the restriction $H(t)$ to $\text{Term}(L)$ is given by 
\begin{enumerate}
    \item[(iii)] $H(t)\mathrel{\mathop:}=x_i$, if $t=x_i\in \text{Var}$;
    \item[(iv)] $H(t)=h(f)[H(t_0)/x_0,\ldots,H(t_{n-1})/x_{n-1}]$, if $t=f(t_0,\ldots,t_{n-1})$ with $f\in F_n$ and $t_0,\ldots,t_{n-1}\in \text{Term}(L)$,
\end{enumerate} 
and the restriction $H(\varphi)$ to $\text{Form}(L)$ is defined to be
\begin{enumerate}
    \item[(v)] $H(\varphi)\mathrel{\mathop:}=(H(t)=H(s))$, if $\varphi$ is the $\mathcal{L}$-atomic formula $(t=s)$ with $s,t\in \text{Term}(L)$;
    \item[(vi)] $H(\varphi)\mathrel{\mathop:}= h(R)[H(t_0)/x_0,\ldots,H(t_{n-1})/x_{n-1}]$, if $\varphi$ denotes the $\mathcal{L}$-atomic formula $R(t_0,\ldots,t_{n-1})$ with $R\in R_n$ and $t_0,\ldots,t_{n-1}\in \text{Term}(L)$;
    \item[(vii)] $H(\varphi)\mathrel{\mathop:}=\neg H(\phi)$, if $\varphi$ is the $\mathcal{L}$-formula $\neg \phi$ with $\phi\in \text{Form}(L)$;
    \item[(viii)]$H(\varphi)\mathrel{\mathop:}=H(\phi)\vee H(\psi)$, if $\varphi$ is the $\mathcal{L}$-formula $\phi\vee \psi$ with $\phi,\psi\in \text{Form}(L)$;
    \item[(ix)] $H(\varphi)\mathrel{\mathop:}=\exists x\, H(\phi)$, if $\varphi$ is the $\mathcal{L}$-formula $\exists x\, \phi$ with $\phi\in \text{Form}(L)$ and $x$ a variable symbol in $\text{Var}$.
\end{enumerate}

Observe that $FV(\varphi)=FV(H(\varphi))$, where $FV(\varphi)$ denotes the set of all free variables occurring in $\varphi$.

The composition rule in \textbf{FOL} is given in the most natural way. Indeed, for any language morphisms $H\colon \mathcal{L}\to \mathcal{L'}$ and $H'\colon \mathcal{L}'\to \mathcal{L}''$,  the map $H'\circ H\colon \mathcal{L}\to \mathcal{L}''$ is the language morphism obtained by extending to $\mathcal{L}$, as above, the following associations: for all $n\geq 0$
\begin{itemize}
    \item $f\mapsto H'(h(f))$, $f\in F_n$,
    \item $R\mapsto H'(h(R))$, $R\in R_n$,
\end{itemize}
where $H'$ is the extension to $\mathcal{L}'$ of $h$. The identity element with respect to $\circ$ is the language morphism $1\colon \mathcal{L}\to \mathcal{L}$ obtained from the extension of the rules: for all $n\geq 0$
\begin{itemize}
    \item $f\mapsto f(x_0,\ldots,x_{n-1})$, $f\in F_n$;
    \item $R\mapsto R(x_0,\ldots,x_{n-1})$, $R\in R_n$.
\end{itemize}
In other words, $1\colon \mathcal{L}\to \mathcal{L}$ is the map which associates each $\mathcal{L}$-term to itself, and each $\mathcal{L}$-formula to itself. 
It is not hard to see that $\circ$ and $1$ satisfy the associativity and identity laws. Therefore, \textbf{FOL} is indeed a category.

Note that \textbf{FOL} has a subcategory of ``simple morphisms'' given by $f \in F_n \mapsto f'(x_0, \cdots, x_{n-1})$, $f' \in F'_n$
and  $R \in R_n \mapsto R'(x_0, \cdots, x_{n-1})$, $R' \in R'_n$.

Here and throughout ``language morphism'' will mean ``a morphism constructed in (i)-(ix)'', unless otherwise stated.

\section{Categories of o-minimal structures}

Throughout this section we fix an order relation symbol $<$.
\\

For each language $\mathcal{L}$, $\mathcal{L}_{<}$ stands for its extension $\text{Term}(L\cup \{<\})\sqcup \text{Form}(L\cup \{<\})$, which is an object in $\textbf{FOL}$. Similarly, any morphism $H\colon \mathcal{L}\to \mathcal{L}'$ in $\textbf{FOL}$ can be extended to a morphism $H_{<}\colon \mathcal{L}_{<}\to \mathcal{L}'_{<}$ in $\textbf{FOL}$ as defined in the previous section. Such a morphism $H_{<}$ is the unique language morphism from $\mathcal{L}_{<}$ to $\mathcal{L}'_{<}$ satisfying the equality $H_{<}\circ \imath=\imath'\circ H$, where $\imath\colon \mathcal{L}\to \mathcal{L}_{<}$ and $\imath'\colon \mathcal{L}'\to \mathcal{L}'_{<}$ indicate the inclusion maps.

As usual we denote the category of all locally small categories by $\textbf{CAT}$. The category $\textbf{Str}(\mathcal{L})$ of all ${L}$-structures whose morphisms are the homomorphisms between ${L}$-structures is an object from $\textbf{CAT}$. A (non full) subcategory of $\textbf{Str}(\mathcal{L})$ is the category $\textbf{Str}_{\textbf{e}}(\mathcal{L})$ of all $L$-structures whose morphisms are the elementary homomorphisms (hence embeddings) between $L$-structures. 
We denote by $\textbf{Str}_{\textbf{omin}}(\mathcal{L}_{<})$ the full (small) subcategory of $\textbf{Str}(\mathcal{L}_{<})$ whose objects are the o-minimal ($L\cup \{<\}$)-structures.  

\begin{definition}[Induced functor] \label{functor}
In view of this discussion, we can form the following contravariant functor $\mathcal{E}\colon \textbf{FOL}\to \textbf{CAT}$:
$$
\mathcal{L}\mapsto \textbf{Str}(\mathcal{L})
$$
and
$$
\mathcal{L}\xrightarrow{H} \mathcal{L}'\mapsto  \textbf{Str}(\mathcal{L})\xleftarrow{\mathcal{E}(H)} \textbf{Str}(\mathcal{L}'),
$$
where $\mathcal{E}(H)$ is the functor given by:
\\

\noindent$\bullet$   $\mathcal{M}'\in \text{Ob}(\textbf{Str}(\mathcal{L}'))\mapsto \mathcal{M}\mathrel{\mathop:}=\mathcal{E}(H)(\mathcal{M}')$, with $|\mathcal{M}|\mathrel{\mathop:}=|\mathcal{M}'|\mathrel{\mathop:}=M'$, and for each $f\in F_n$ and each $R\in R_n$ we have $f^{\mathcal{M}}\mathrel{\mathop:}=H(f)^{\mathcal{M}'}\colon M'^n\to M'$ (that is, $f^\mathcal{M}$ is the interpretation of the $L'$-term $H(f)$ in $\mathcal{M}'$) and $R^{\mathcal{M}}\mathrel{\mathop:}=H(R)^{\mathcal{M}'}\subseteq M'^n$ (that is, $R^{\mathcal{M}}$ is the interpretation of the atomic $L'$-formula $H(R)$ in $\mathcal{M}'$). Thus, for any $L$-formula $\varphi(x_0,\ldots,x_{n-1})$ and any valuation $\nu\colon \{x_0,\ldots,x_{n-1}\}\to M'$ we obtain 
    \begin{empheq}[box=\ovalbox]{align*}
    (*)\ \mathcal{M}\models_{\nu}\varphi(x_0,\ldots,x_{n-1})\ \text{if and only}\ \mathcal{M}'\models_{\nu} H(\varphi)(x_0,\ldots,x_{n-1}),
   \end{empheq}
    by induction on the complexity of $\varphi$.
\\

\noindent$\bullet$   $\alpha'\in \text{Hom}_{\textbf{Str}(\mathcal{L}')}(\mathcal{M}'_1,\mathcal{M}'_2)\mapsto \mathcal{E}(H)(\alpha')\mathrel{\mathop:}=\alpha'\in \text{Hom}_{\textbf{Str}(\mathcal{L})}(\mathcal{M}_1,\mathcal{M}_2)$.
\end{definition}

\noindent\textbf{Remark 1.}\label{remarkindfunctor}
There are some variants of the functor $\mathcal{E}$, namely:\footnote{Clearly, other similar contravariant functors can be defined, corresponding to other kinds of morphisms between structures.}
\begin{enumerate}
\item[(a)] the contravariant functor $\mathcal{E}_{\textbf{e}}\colon \textbf{FOL}\to \textbf{CAT}$ given by
$$
\mathcal{L}\mapsto \textbf{Str}_{\textbf{e}}(\mathcal{L})
$$
and
$$
\mathcal{L}\xrightarrow{H} \mathcal{L}'\mapsto  \textbf{Str}_{\textbf{e}}(\mathcal{L})\xleftarrow{\mathcal{E}_e(H)} \textbf{Str}_{\textbf{e}}(\mathcal{L}'),
$$
where $\mathcal{E}_e(H)$ is defined the same way as above for the category $\textbf{Str}(\mathcal{L}')$. It is worth noticing that $\alpha'\in \text{Hom}_{\textbf{Str}_{\textbf{e}}(\mathcal{L}')}(\mathcal{M}'_1,\mathcal{M}'_2)\mapsto \mathcal{E}(H)(\alpha')\mathrel{\mathop:}=\alpha'\in \text{Hom}_{\textbf{Str}_{\textbf{e}}(\mathcal{L})}(\mathcal{M}_1,\mathcal{M}_2)$ is well defined by virtue of ($*$).
\item[(b)] the contravariant functor $\mathcal{E}_<\colon \textbf{FOL}\to \textbf{CAT}$ given by 
$$
\mathcal{L}_<\mapsto \textbf{Str}(\mathcal{L}_<)
$$
and
$$
\mathcal{L}_<\xrightarrow{H_<} \mathcal{L}'_<\mapsto  \textbf{Str}(\mathcal{L}_<)\xleftarrow{\mathcal{E}_<(H_<)} \textbf{Str}(\mathcal{L}'_<),
$$
where $\mathcal{E}_<(H_<)$ is defined analogously to $\mathcal{E}(H)$. 
\end{enumerate}

\begin{theorem}The functor $\mathcal{E}_<(H_<)\colon \textbf{Str}(\mathcal{L}'_{<})\to \textbf{Str}(\mathcal{L}_{<})$ (see Remark 1(b)) maps o-minimal structures in the language $\mathcal{L}'_{<}$ to o-minimal structures in the language $\mathcal{L}_{<}$, in other words, the following diagram commutes
$$
\xymatrix{
\textbf{Str}(\mathcal{L}'_{<})\ar[r]^{\mathcal{E}_<(H_{<})} & \textbf{Str}(\mathcal{L}_{<})\\
\textbf{Str}_{\textbf{omin}}(\mathcal{L}'_{<})\ar@{^{(}->}[u]\ar[r]_{\mathcal{E}_<(H_{<})|} &\textbf{Str}_{\textbf{omin}}(\mathcal{L}_{<})\ar@{^{(}->}[u]
} 
$$
where $\mathcal{E}_<(H_{<})|$ denotes the restriction of $\mathcal{E}_<(H_{<})$ to the subcategory $\textbf{Str}_{\textbf{omin}}(\mathcal{L}_{<})$.
\end{theorem}
\begin{proof}
It follows immediately from ($*$) and the fact $FV(\varphi)=FV(H_<(\varphi))$, for any first order formula $\varphi$ in $\mathcal{L}_<$.
\end{proof}

The dichotomy result stated in Fact 1 (see Section 1) can be translated in this section into diagrams of categories of o-minimal structures and functors induced by language morphisms: 

\begin{figure}[h]
\[
\xymatrix{
       &\widetilde{\mathcal{L}}\\
\mathcal{L}_{or}\ar@{.>}[ru] & \mathcal{L}_{og}\ar@{_{(}->}[l]^{\jmath}\ar@{^{(}->}[u]_{\imath}
}
\]
\caption{Diagram in \textbf{FOL}}\label{Lang1}
\end{figure}
\noindent where $\mathcal{L}_{or}$ is the language generated by the signature of the ordered rings $L_{or}$, $\mathcal{L}_{og}$ is the language generated by the signature of the ordered groups $L_{og}$ and $\widetilde{\mathcal{L}}$ expands  $\mathcal{L}_{og}$ arbitrarily. Applying the functor $\mathcal{E}_<$, we get

\begin{figure}[h]
\[
\xymatrix{
 & \widetilde{R}\ar[d]^{\mathcal{E}_<(\imath)}\ar@{.>}[ld]\\
(R,<,+,\cdot)\ar[r]_{\mathcal{E}_<(\jmath)} & (R,<,+)
}
\]
\caption{Diagram in \textbf{CAT}}\label{Struc1}
\end{figure}
\noindent where $\widetilde{R}$ is an o-minimal expansion of ordered group $(R,<,+)$.

Similarly, the dichotomy stated in Fact 2 in Section 1 can be read out of the following diagrams: 

\begin{figure}[h]
\[
\xymatrix{
       &\widetilde{\mathcal{L}}\\
\mathcal{L}_{exp}\ar@{.>}[ru] & \mathcal{L}_{or}\ar@{_{(}->}[l]^{\jmath}\ar@{^{(}->}[u]_{\imath}
}
\]
\caption{Diagram in \textbf{FOL}}\label{Lang2}
\end{figure}
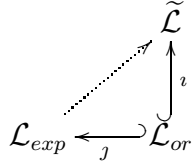
\noindent where $\mathcal{L}_{exp}$ is the language generated by the signature $L_{or}\cup \{\text{exp}\}$ and $\widetilde{\mathcal{L}}$ expands  $\mathcal{L}_{or}$ arbitrarily, and 

\begin{figure}[h]
\[
\xymatrix{
 & \widetilde{R}\ar[d]^{\mathcal{E}_<(\imath)}\ar@{.>}[ld]\\
\overline{\mathbb{R}}_{\text{exp}}\ar[r]_{\mathcal{E}_<(\jmath)} & \overline{\mathbb{R}}
}
\]
\caption{Diagram in \textbf{CAT}}\label{Struc2}
\end{figure}
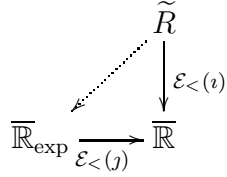
\noindent where $\overline{\mathbb{R}}$ stands for the ordered field of real numbers, $\overline{\mathbb{R}}_{\text{exp}}$ is the exponential real field $(\overline{\mathbb{R}},\exp)$ and  $\widetilde{R}$ is an o-minimal expansion of $\overline{\mathbb{R}}$.

Observe that the dichotomy theorems in Facts 1 and 2 characterize the images of the induced functors as considered above (Definition \ref{functor}).

The above remarks suggest that may be useful to consider the following notion:

\begin{definition}
We define the category {\bf STR} of all structures by means of the Grothendieck construction as follows. 
\begin{itemize}
\item Ob({\bf STR}): $(\mathcal{L},\mathcal{M})$, where $\mathcal{L}$ is a language and $\mathcal{M}\in \textbf{Str}(\mathcal{L})$;
\item For any pair $(\mathcal{L},\mathcal{M})$ and $(\mathcal{L}',\mathcal{N}')$, $\text{Hom}_{\bf STR}((\mathcal{L},\mathcal{M}), (\mathcal{L}',\mathcal{N}'))$ is the set of pairs $(H,\alpha)$ where  $H\colon \mathcal{L} \to \mathcal{L}'$ is a language morphism and $\alpha\colon \mathcal{E}(H)(\mathcal{N}')\to \mathcal{M}$ is a morphism in $\textbf{Str}(\mathcal{L}')$;
\item Composition: $(H', \alpha')\ast (H, \alpha) \mathrel{\mathop:}= (H' \circ H, \alpha \circ \mathcal{E}(H)(\alpha'))$;
\item Identities: $\text{id}_{(\mathcal{L},\mathcal{M})}\mathrel{\mathop:}= (\text{id}_{\mathcal{L}}, \text{id}_{\mathcal{M}})$.
\end{itemize}
We have some variants of $\textbf{STR}$ such as: 
\begin{itemize}
    \item[(a)] $\textbf{STR}_{\textbf{e}}$, where $\alpha$ as in $\textbf{STR}$ are taken to be elementary homomorphims;
    \item[(b)] $\textbf{STR}_{\textbf{e}_1}$, where $\alpha$ as in $\textbf{STR}$ preserve only the validity of first order unary formulas;
    \item[(c)] $\textbf{STR}_<$ constructed analogously to $\textbf{STR}$ for all language expansions $\mathcal{L}_<$;  
    \item[(d)] $\textbf{STR}_{\textbf{o-min}}$ as $\textbf{STR}_<$, with $\mathcal{M}\in \textbf{Str}_{\textbf{o-min}}(\mathcal{L}_<)$.
\end{itemize}
\end{definition}

Note that the dichotomy results expressed in Facts 1 and 2 can also be read in this global context, since the morphism from $(\mathcal{L}_<,\mathcal{M})$ to $(\mathcal{L}'_<,\mathcal{N}')$ is the pair $H_<\colon \mathcal{L}_<\to \mathcal{L}'_<$ and $\alpha\colon \mathcal{E}(H_<)(\mathcal{N}')\to \mathcal{M}$ is the identity homomorphism, that is,   $\mathcal{E}(H_<)(\mathcal{N}')=\mathcal{M}$ and $\alpha=\text{id}_{\mathcal{M}}$.

On the other hand, a more general case in which the map $\alpha$ is not necessarily the identity also occurs in the literature. For instance, 
\\

\noindent\textbf{Fact 3} (\cite{ressayre}). If $\mathcal{M}$ is any nonstandard model of PA, with $(\text{HF}^{\mathcal{M}}, \in^{\mathcal{M}})$ the corresponding nonstandard hereditary finite sets of $\mathcal{M}$ (by Ackerman coding: the natural numbers of $\text{HF}^{\mathcal{M}}$ are
isomorphic to $\mathcal{M}$), then for any consistent computably axiomatized theory $T$ extending ZF in the language of set theory, there is a submodel $\mathcal{N}'\subseteq (\text{HF}^{\mathcal{M}}, \in^{\mathcal{M}})$ such that $\mathcal{N}'\models T$.
\\

\section{Final remarks}

\noindent$\bullet$ It is natural to consider even more general forms of induced functors by changing of languages as in \cite{visser}: for instance, something in this direction already occurred in Facts 1 (and 2) since  $\cdot$ is $\widetilde{L}$ definable in $\widetilde{R}$. This would complete the picture of Facts 1, 2 (that is, it would name the dot arrows in the diagrams shown in Figures \ref{Lang1}, \ref{Struc1}, \ref{Lang2} and \ref{Struc2}). 
\\

\noindent$\bullet$ Are there natural examples of the phenomenon appeared in Fact 3 in the setting of o-minimal structures? That is, a situation involving o-minimal structures and a morphism from  $(\mathcal{L}_<,\mathcal{M})$ to $(\mathcal{L}'_<,\mathcal{N}')$, which is the $H_<\colon \mathcal{L}_<\to \mathcal{L}'_<$ and $\alpha\colon \mathcal{E}(H_<)(\mathcal{N}')\to \mathcal{M}$, where $\mathcal{E}(H_<)(\mathcal{N}')\neq \mathcal{M}$ and/or $\alpha \neq \text{id}_{\mathcal{M}}$. What about with $\alpha$ being an embedding? Or an elementary embedding? Or an $e_1$-elementary embedding, that is, an embedding which preserves formulas with one free variable?

\end{document}